\documentclass[11pt]{article}

\input diagram.tex
\usepackage{amsmath} 
\usepackage[mathscr]{eucal}
\usepackage{amssymb} 
\usepackage{theorem}
\usepackage{tikz-cd}
\usetikzlibrary{positioning}

\usepackage{enumerate}

\theoremstyle{change}  
\newtheorem{theorem}{Theorem}[section] 
\newtheorem{lemma}[theorem]{Lemma}  

\newtheorem{corollary}[theorem]{Corollary}
\theorembodyfont{\rmfamily}  

\newtheorem{notation}[theorem]{Notation}

\newenvironment{proof}{\noindent{\bf Proof}\ }{\qed\bigskip}

\renewcommand{\le}{\leqslant} 

\newcommand{\Aut}{\mathrm{Aut}}

\newcommand{\calF}{\mathcal{F}}

\newcommand{\calP}{\mathcal{P}}

\newcommand{\calQ}{\mathcal{Q}}

\newcommand{\catfont}{\mathsf}

\newcommand{\FF}{\mathbb{F}}

\newcommand{\id}{\mathrm{id}}

\newcommand{\Inn}{\mathrm{Inn}}

\newcommand{\lmod}[1]{\llap{\phantom{|}}_{#1}\catfont{mod}}
\newcommand{\lMod}[1]{\llap{\phantom{|}}_{#1}\catfont{Mod}}

\newcommand{\Out}{\mathrm{Out}}

\newcommand{\qed}{\nobreak\hfill
                  \vbox{\hrule\hbox{\vrule\hbox to 5pt
                  {\vbox to 8pt{\vfil}\hfil}\vrule}\hrule}}

\newcommand{\ZZ}{\mathbb{Z}}

\newcommand{\DD}{D^{\Delta}}

\newcommand{\TD}{T^{\Delta}}

\newcommand{\FFTD}{\mathbb{F}T^{\Delta}}
\newcommand{\RTD}{RT^{\Delta}}

\newcommand{\Proj}{\mathrm{Proj}}
\newcommand{\Fppk}[1]{\calF_{#1pp_k}^\Delta}
\newcommand{\Mult}{\mathrm{Mult}}
\newcommand{\Sym}{\mathrm{Sym}}
\newcommand{\Alt}{\mathrm{Alt}}
\title{On functorial equivalence classes of blocks of finite groups}
\author{Deniz Y\i lmaz}
\date{}
\providecommand{\keywords}[1]
{
  \small\smallskip\par	
  \hspace{2ex}\textbf{Keywords:} #1
}
\providecommand{\msc}[1]
{
  \small\smallskip\par	
  \hspace{2ex}\textbf{MSC2020:} #1
}

\begin{document}
\sloppy
\maketitle
\begin{abstract}
Let $k$ be an algebraically closed field of characteristic $p>0$ and let $\FF$ be an algebraically closed field of characteristic $0$. Recently, together with Bouc, we introduced the notion of functorial equivalences between blocks of finite groups and proved that given a $p$-group $D$, there is only a finite number of pairs $(G,b)$ of a finite group $G$ and a block $b$ of $kG$ with defect groups isomorphic to $D$, up to functorial equivalence over $\FF$. In this paper, we classify the functorial equivalence classes over $\FF$ of blocks with cyclic defect groups and $2$-blocks of defects $2$ and $3$. In particular, we prove that for all these blocks, the functorial equivalence classes depend only on the fusion system of the block.
\end{abstract}

\msc{16S34, 20C20, 20J15.}

\section{Introduction}
Throughout of the paper, $k$ denotes an algebraically closed field of characteristic $p>0$ and $\FF$ denotes an algebraically closed field of characteristic zero. The local-global phenomena in modular representation theory of finite groups asserts that the global invariants of blocks are determined by the local invariants. There are many outstanding conjectures that revolves around this principle. One such conjecture is Puig's finiteness conjecture which asserts that given a finite $p$-group $D$, there are only finitely many pairs $(G,b)$ of a finite group $G$ and a block idempotent $b$ of $kG$ with defect group $D$, up to splendid Morita equivalence (Conjecture 6.4.2 in~\cite{linckelmann2018}). Splendidly Morita equivalent blocks have isomorphic source algebras, and hence Puig's conjecture, if true, means that all the global invariants of a block are determined by the defect group up to finitely many possibilities.

In \cite{BoucYilmaz2022}, together with Bouc, we introduced the notion of functorial equivalences over $\FF$ between blocks of finite groups, weaker than splendid Morita equivalence, and proved the following finiteness theorem.

\begin{theorem}\cite[Theorem~ 10.6]{BoucYilmaz2022}
Given a finite $p$-group $D$, there is only a finite number of pairs $(G,b)$, where $G$ is a finite group and $b$ is a block idempotent of $kG$ with defect group $D$, up to functorial equivalence over $\FF$.
\end{theorem}

To prove Puig's conjecture, it suffices to show that for a given $p$-group $D$ every functorial equivalence class of blocks with defect $D$ is a union of finitely many splendid Morita equivalence classes. Therefore, it is a natural question to classify the functorial equivalence classes of blocks with a given defect group $D$. In this paper, we start the program of classifying the functorial equivalence classes of blocks and consider the cases where $D\in \{C_{p^n}, V_4, Q_8,D_8, C_2\times C_2\times C_2, C_2\times C_4\}$. We summarize our result as follows. For a finite group $G$ we denote by $b_0(G)$ the principal block of $kG$.

\begin{theorem}\label{thm main}
Let $G$ be a finite group and let $b$ be a block idempotent of $kG$ with a defect group $D$.

\smallskip
{\rm (a)} The functorial equivalence classes over $\FF$ of blocks with cyclic defect groups depend only on the inertial quotient of the blocks. In particular,  for blocks with cyclic defect groups the functorial equivalence classes over $\FF$ coincide with the splendid Rickard equivalence classes.

\smallskip
{\rm (b)} If $D=V_4$, then the pair $(G,b)$ is functorially equivalent over $\FF$ to either $(V_4,1)$ or $(A_4,1)$.  In particular,  for blocks with Klein four defect groups the functorial equivalence classes over $\FF$ coincide with the splendid Rickard equivalence classes.

\smallskip
{\rm (c)} If $D=Q_8$, then the pair $(G,b)$ is functorially equivalent over $\FF$ to either $(Q_8,1)$ or $(SL(2,3),b_0(SL(2,3))$.

\smallskip
{\rm (d)} If $D=D_8$, then then the pair $(G,b)$ is functorially equivalent over $\FF$ to $(D_8,1)$, $(S_4,b_0(S_4))$ or $(PSL(3,2),b_0(PSL(3,2))$.

\smallskip
{\rm (e)} If $D=C_2\times C_2\times C_2$, then the pair $(G,b)$ is functorially equivalent over $\FF$ to $(C_2\times C_2\times C_2,1)$, $(A_4\times C_2,1)$, $(SL_2(8),b_0(SL_2(8)))$ or $(J_1,b_0(J_1))$. In particular, for bloks with defect groups $C_2\times C_2\times C_2$ the functorial equivalence classes over $\FF$ coincide with the isotypy classes.

\smallskip
{\rm (f)} If $D=C_2\times C_4$, then the pair $(G,b)$ is functorially equivalent over $\FF$ to $(C_2\times C_4,1)$.
\end{theorem}

Theorem~\ref{thm main} follows from the more precise Theorems~\ref{thm cyclic}, \ref{thm Kleinfour}, \ref{thm quaternion}, \ref{thm dihedral}, \ref{thm elementaryabelian} and \ref{thm abelinoforder8}. The following corollary is immediate from Theorem~\ref{thm main}.

\begin{corollary}
Functorial equivalence classes over $\FF$ of blocks of finite groups with defect groups $D\in  \{C_{p^n}, V_4, Q_8,D_8, C_2\times C_2\times C_2, C_2\times C_4\}$ depend only on the fusion system of the blocks.
\end{corollary}

We also find the composition factors of the diagonal $p$-permutation functors arising from all these blocks except when $D=C_2\times C_2\times C_2$.

In Section~\ref{sec prelim} we recall diagonal $p$-permutation functors and functorial equivalences of blocks. We consider blocks with cyclic defect groups in Section~\ref{sec cyclic}, with Klein four defect groups in Section~\ref{sec Kleinfour}, with $Q_8$ defect groups in Section~\ref{sec quaternion}, with $D_8$ defect groups in Section~\ref{sec dihedral}, with $C_2\times C_2\times C_2$ defect groups in Section~\ref{sec elementaryabelian} and with $C_2\times C_4$ defect groups in Section~\ref{sec otherabelian}.

\section{Preliminaries}\label{sec prelim}
{\rm (a)} Let $(P,s)$ be a pair where $P$ is a $p$-group and $s$ is a generator of a $p'$-group acting on $P$. We write $P\langle s\rangle := P\rtimes \langle s\rangle$ for the corresponding semi-direct product. We say that two pairs $(P,s)$ and $(Q,t)$ are {\em isomorphic} and write $(P,s)\cong (Q,t)$, if there is a group isomorphism $f: P\langle s\rangle \to Q\langle t\rangle$ that sends $s$ to a conjugate of $t$. We set $\Aut(P,s)$ to be the group of automorphisms of the pair $(P,s)$ and $\Out(P,s)=\Aut(P,s)/\Inn(P\langle s\rangle)$. Recall from \cite{BoucYilmaz2020} that a pair $(P,s)$ is called a {\em $\DD$-pair}, if $C_{\langle s\rangle}(P)=1$. See also \cite[Lemma~6.10]{BoucYilmaz2022}.

\smallskip
\noindent{\rm (b)} Let $G, H$ and $K$ be finite groups. We call a $(kG,kH)$-bimodule $M$ a \textit{diagonal $p$-permutation bimodule}, if $M$ is a $p$-permutation $(kG,kH)$-bimodule whose indecomposable direct summands have twisted diagonal vertices as subgroups of $G\times H$, or equivalently, if $M$ is a $p$-permutation $(kG,kH)$-bimodule which is projective both as a left $kG$-module and as a right $kH$-module. We denote by $\TD(kG,kH)$ the Grothendieck group of diagonal $p$-permutation $(kG,kH)$-bimodules. For a commutative ring $R$, we also set $R\TD(kG,kH):=R\otimes_{\ZZ}T^{\Delta}(kG,kH)$. If $b$ is a block idempotent of $kG$ and $c$ a block idempotent of $kH$, then we define $\TD(kGb,kHc)$ and $R\TD(kGb,kHc)$ similarly. 

\smallskip
If $M$ is a diagonal $p$-permutation $(kG,kH)$-bimodule and $N$ a diagonal $p$-permutation $(kH,kK)$-bimodule, then the tensor product $M\otimes_{kH} N$ is a diagonal $p$-permutation $(kG,kK)$-bimodule. This induces an $R$-linear map
\begin{align*}
\cdot_H: R\TD(kG,kH)\times R\TD(kH,kK)\to R\TD(kG,kK)\,.
\end{align*}

\smallskip
\noindent{\rm (c)} Let $Rpp_k^\Delta$ denote the following category:
\begin{itemize}
\item objects: finite groups.
\item $\mathrm{Mor}_{Rpp_k^\Delta}(G,H) = RT^{\Delta}(kH,kG)$.
\item composition is induced from the tensor product of bimodules.
\item $\mathrm{Id}_{G}=[kG]$.
\end{itemize} 
An $R$-linear functor from $Rpp_k^\Delta$ to $\lMod{R}$ is called a \textit{diagonal $p$-permutation functor} over $R$. Together with natural transformations, diagonal $p$-permutation functors form an abelian category $\Fppk{R}$.

\smallskip
\noindent{\rm (d)} Let $G$ be a finite group and $b$ a block idempotent of $kG$.  Recall from \cite{BoucYilmaz2022} that the block diagonal $p$-permutation functor $R\TD_{G,b}$ is defined as
\begin{align*}
R\TD_{G,b}: R pp_k^{\Delta}&\to \lMod{R}\\
H&\mapsto R T^\Delta(kH,kG)\otimes_{kG} kGb\,.
\end{align*}

If $H$ is another finite group and if $c$ is a block idempotent of $kH$, we say that the pairs $(G,b)$ and $(H,c)$ are {\em functorially equivalent} over $R$, if the corresponding diagonal $p$-permutation functors $R\TD_{G,b}$ and $R\TD_{H,c}$ are isomorphic in $\Fppk{R}$ (\cite[Definition~10.1]{BoucYilmaz2022}).  By \cite[Lemma~10.2]{BoucYilmaz2022} the pairs $(G,b)$ and $(H,c)$ are functorially equivalent over~$R$ if and only if there exist $\omega\in R\TD(kGb,kHc)$ and $\sigma\in R\TD(kHc,kGb)$ such that
\begin{align*}
\omega \cdot_G \sigma = [kGb] \quad \text{in} \quad \RTD(kGb,kGb) \quad \text{and} \quad \sigma \cdot_H \omega = [kHc] \quad \text{in} \quad \RTD(kHc,kHc) \,.
\end{align*}
Note that this implies that a $p$-permutation equivalence between blocks implies a functorial equivalence over $\ZZ$ and hence a functorial equivalence over $R$, for any $R$.

\smallskip
\noindent{\rm (e)} Recall from \cite{BoucYilmaz2022} that the category $\Fppk{\FF}$ is semisimple. Moreover, the simple diagonal $p$-permutation functors $S_{L,u,V}$, up to isomorphism, are parametrized by the isomorphism classes of triples $(L,u,V)$ where $(L,u)$ is a $\DD$-pair, and $V$ is a simple $\FF\Out(L,u)$-module (see \cite[Sections~6~and~7]{BoucYilmaz2022} for more details on simple functors).

\smallskip
\noindent{\rm (f)} Since the category $\Fppk{\FF}$ is semisimple, the functor $\FFTD_{G,b}$ is a direct sum of simple diagonal $p$-permutation functors $S_{L,u,V}$.  Hence two pairs $(G,b)$ and $(H,c)$ are functorially equivalent over $\FF$ if and only if for any triple $(L,u,V)$, the multiplicities of the simple diagonal $p$-permutation functor $S_{L,u,V}$ in $\FFTD_{G,b}$ and $\FFTD_{H,c}$ are the same.  We now recall the formula for the multiplicity of $S_{L,u,V}$ in $\FFTD_{G,b}$.  See \cite[Section~8]{BoucYilmaz2022} for more details. 

Let $(D,e_D)$ be a maximal $kGb$-Brauer pair.  For any subgroup $P\le D$, let $e_P$ be the unique block idempotent of $kC_G(P)$ with $(P,e_P)\le (D,e_D)$ (see, for instance,  \cite[Section~6.3]{linckelmann2018} for more details on Brauer pairs). Let also $\calF_b$ be the fusion system of $kGb$ with respect to $(D,e_D)$ and let $[\calF_b]$ be a set of isomorphism classes of objects in $\calF_b$.

For $P\in \calF_b$, we set $\calP_{(P,e_P)}(L,u)$ to be the set of group isomorphisms $\pi:L\to P$ with $\pi i_u \pi^{-1}\in \Aut_{\calF_b}(P)$. The set $\calP_{(P,e_P)}(L,u)$ is an $(N_G(P,e_P), \Aut(L,u))$-biset via
\begin{align*}
g\cdot \pi\cdot\varphi = i_g\pi\varphi
\end{align*}
for $g\in N_G(P,e_P)$, $\pi\in \calP_{(P,e_P)}(L,u)$ and $\varphi\in\Aut(L,u)$. We denote by $[\calP_{(P,e_P)}(L,u)]$ a set of representatives of $N_G(P,e_P)\times \Aut(L,u)$-orbits of $\calP_{(P,e_P)}(L,u)$. 

For $\pi \in [\calP_{(P,e_P)}(L,u)]$, the stabilizer in $\Aut(L,u)$ of the $N_G(P,e_P)$-orbit of $\pi$ is denoted by $\Aut(L,u)_{\overline{(P,e_P,\pi)}}$. One has
\begin{align*}
\Aut(L,u)_{\overline{(P,e_P,\pi)}}=\{\varphi\in \Aut(L,u)\mid \pi\varphi \pi^{-1}\in \Aut_{\calF_b}(P)\}\,.
\end{align*}

\begin{theorem}\cite[Theorem~8.22(b)]{BoucYilmaz2022}\label{thm multiplicityformula}
The multiplicity of a simple diagonal $p$-permutation functor $S_{L,u,V}$ in the functor $\FFTD_{G,b}$ is equal to the $\FF$-dimension of
\begin{align*}
\bigoplus_{P \in [\calF_b]} \bigoplus_{\pi \in [\calP_{(P,e_P)}(L,u)]} \FF\Proj(ke_PC_G(P),u)\otimes_{\Aut(L,u)_{\overline{(P,e_P,\pi)}}} V\,.
\end{align*}
\end{theorem}

Let $G$ be a finite group. We denote by $\calQ_{G,p}$ the set of pairs $(P,s)$ where $P$ is a $p$-subgroup of $G$ and $s$ is a $p'$-element of $N_G(P)$. The group $G$ acts on $\calQ_{G,p}$ via conjugation and we denote by $[\calQ_{G,p}]$ a set of representatives of the $G$-orbits on $\calQ_{G,p}$.  

If $(P,s)\in \calQ_{G,p}$, then the pair $(\tilde{P},\tilde{s}):=(PC_{\langle s\rangle}(P)/C_{\langle s\rangle}(P), sC_{\langle s\rangle}(P))$ is a $\DD$-pair. Suppose that $(L,u)$ is another $\DD$-pair isomorphic to $(\tilde{P},\tilde{s})$. Then the isomorphism between the pairs induces a group homomorphism from $N_G(P,s)$ to $\Out(L,u)$, see \cite[Section~7]{BoucYilmaz2022}. So, a simple $\FF\Out(L,u)$-module $V$ can be viewed as an $\FF N_G(P,s)$-module via this homomorphism.

\begin{theorem}\cite[Corollary~7.4]{BoucYilmaz2022}\label{thm multiplicityinrepresentable}
The multiplicity of a simple diagonal $p$-permutation functor $S_{L,u,V}$ in the representable functor $\FFTD_{G}$ is equal to the $\FF$-dimension of
\begin{align*}
\bigoplus_{\substack{(P,s)\in[\mathcal{Q}_{G,p}]\\ (\tilde{P},\tilde{s})\cong (L,u)}} V^{N_G(P,s)}\,.
\end{align*}
\end{theorem}

\begin{notation}
Let $G$ be a finite group and let $b$ be a block idempotent of $kG$.

\smallskip
{\rm (a)} We denote the multiplicity of a simple diagonal $p$-permutation functor $S_{L,u,V}$ in $\FFTD_{G,b}$ by $\Mult(S_{L,u,V}, \FFTD_{G,b})$.

{\rm (b)} We denote by $l(kGb)$ the number of isomorphism classes of simple $kGb$-modules. By \cite[Corollary~8.23]{BoucYilmaz2022}, one has $\Mult(S_{1,1,\FF}, \FFTD_{G,b})=l(kGb)$.
\end{notation}

The following lemma will be used in Sections \ref{sec quaternion} and \ref{sec dihedral}.

\begin{lemma}\label{lem multiplicityoftop}
Let $G$ be a finite group and let $b$ be a block idempotent of $kG$ with a defect group $D$. Let $(D,e_D)$ be a maximal $b$-Brauer pair and let $\calF_b$ be the fusion system of $b$ with respect to $(D,e_D)$. Let $\overline{\Aut_{\calF_b}(D)}$ denote the image of $\Aut_{\calF_b}(D)$ in $\Out(D)$. Then for any simple $\FF\Out(D)$-module $V$, we have
\begin{align*}
\Mult(S_{D,1,V},\FFTD_{G,b})=\dim_\FF \left(V^{\overline{\Aut_{\calF_b}(D)}}\right)\,.
\end{align*}
\end{lemma}
\begin{proof}
One shows that
\begin{align*}
\calP_{(D,e_D)}(D,1)=\Aut(D)\quad \text{and}\quad [N_G(D,e_D)\backslash \calP_{(D,e_D)}(D,1)/\Aut(D)]=[\id_D]\,.
\end{align*}
Moreover,
\begin{align*}
\Aut(D)_{\overline{(D,e_D,\id_D)}}=\Aut_{\calF_b}(D)\,.
\end{align*}

Since $kC_G(D)e_D$ has a central defect group $Z(D)$, it has a unique isomorphism class of simple modules and hence
\begin{align*}
\FF\Proj(kC_G(D)e_D,1)\cong \FF\,.
\end{align*}

Theorem~\ref{thm multiplicityformula} implies now that
\begin{align*}
\Mult(S_{D,1,V},\FFTD_{G,b})=\dim_\FF\left(\FF\otimes_{\Aut_{\calF_b}(D)} V\right)=\dim_\FF\left(V^{\Aut_{\calF_b}(D)} \right) =\dim_\FF\left(V^{\overline{\Aut_{\calF_b}(D)}}\right)\,,
\end{align*}
as desired.
\end{proof}

\section{Blocks with cyclic defect groups}\label{sec cyclic}

Let $G$ be a finite group and let $b$ be a block idempotent of $kG$ with a cyclic defect group $D$. We will give a decomposition of the functor $\FFTD_{G,b}$ in terms of the simple diagonal $p$-permutation functors. We refer the reader to \cite[Chapter~11]{linckelmann2018} for more details on blocks with cyclic defect groups. Let $(D,e_D)$ be a maximal $b$-Brauer pair and let $E=N_G(D,e_D)/D$ be the inertial quotient of $b$.  Then for every $b$-Brauer pair $(P,e_P)\le (D,e_D)$ one has $N_G(P,e_P)/C_G(P)\cong E$, see, for instance,  \cite[Theorem~11.2.1]{linckelmann2018}.

First of all, the multiplicity of $S_{1,1,\FF}$ is equal to $l(kGb)$ which is equal to $|E|$ by \cite[Theorem~11.1.3]{linckelmann2018}. Assume now that $L$ is a nontrivial cyclic $p$-group. Then $\Aut(L)$ is an abelian group and hence one can show that for $p'$-elements $u,u'\in \Aut(L)$, the pairs $(L,u)$ and $(L,u')$ are isomorphic if and only if $u=u'$. Moreover, $\Out(L,u)\cong \Aut(L)/\langle u\rangle$ is abelian. 

Let $P\le D$ with $P\cong L$.  We identify $L$ with $P$ and $E$ with its image in $\Aut(P)$ under the map $E\to \Aut(P)$, $s\mapsto i_s$.  Via these identifications we have $E=\Aut_{\calF_b}(P)$.

For any $p'$-element $u\in \Aut(P)$, we have 
\begin{equation*}
\calP_{(P,e_P)}(P,u) =\{\pi\in\Aut(P)|\, \pi i_u\pi^{-1} = i_u \in \Aut_{\calF_b}(P) \} = \left \{
  \begin{aligned}
    &\Aut(P), && \text{if}\ u \in E \\
    &\emptyset, && \text{otherwise}\,.
  \end{aligned} \right.
\end{equation*}
If $u \notin E$, then the simple functor $S_{L,u,V}$ is not a summand of $\FFTD_{G,b}$. If $u \in E$, then $\calP_{(P,e_P)}(P,u)=\Aut(P)$, and hence one can show that there is only one $N_G(P,e_P)\times \Aut(P,u)$-orbit of $\Aut(P)$, i.e., $[\calP_{(P,e_P)}(P,u)]=[\id]$. Moreover, one has
\begin{align*}
\Aut(P,u)_{\overline{(P,e_P,\id)}}=\{\phi\in\Aut(P,u) |\, \exists g\in N_G(P,e_P), i_g=\phi\}=E\,.
\end{align*}

Now since $b$ is a block with cyclic defect group, by \cite[Theorem~11.2.1]{linckelmann2018} the block idempotent $e_P$ of $kC_G(P)$ is nilpotent, and so it has a unique simple module, up to isomorphism. Therefore we have $ \FF\Proj(keC_G(P),u)\cong \FF$, and it follows that the multiplicity of the simple functor $S_{P,u,V}$, for $u\in E$, is equal to the $\FF$-dimension of the fixed points $V^E$. Since $\Out(P,u)$ is abelian, the dimension of $V$ is equal to one and hence $V^E$ is either zero or equal to $V$. We proved the following.

\begin{theorem}\label{thm cyclic}
Let $G$ be a finite group and let $b$ be a block idempotent of $kG$ with a cyclic defect group $D$ and inertial quotient $E$.  Then
\begin{align*}
\FFTD_{G,b}\cong |E| S_{1,1,\FF} \bigoplus_{1<P\le D} \bigoplus_{u\in E} \bigoplus_{\substack{V\in \lmod{\FF[\Out(P,u)/E]}\\ \text{simple}}} S_{P,u,V} \,.
\end{align*}
\end{theorem}

\begin{corollary}
Let $G$ and $H$ be finite groups. Let $b$ be a block idempotent of $kG$ and $c$ a block idempotent of $kH$ with cyclic defect groups isomorphic to $D$. Then $(G,b)$ and $(H,c)$ are functorially equivalent over $\FF$ if and only if the inertial quotients of $b$ and $c$ are isomorphic. In particular, $kGb$ and $kHc$ are splendidly Rickard equivalent if and only if $(G,b)$ and $(H,c)$ are functorially equivalent over $\FF$. 
\end{corollary}
\begin{proof}
The first assertion follows from Theorem~\ref{thm cyclic}, and the second assertion follows from the first one and \cite{Rouquier98}.
\end{proof}

\section{Blocks with Klein four defect groups}\label{sec Kleinfour}

Let $C_2$ denote a cyclic group of order $2$ and let $V_4$ denote a Klein-four group. Since $\Aut(C_2)=\{1\}$, the functor $S_{C_2,1,\FF}$ is the unique simple functor, up to isomorphism, with parametrizing set $(L,u,V)$ where $L\cong C_2$. 

Let $u\in \Aut(V_4)\cong \Sym(3)$ be an element of order $3$. One shows that a $\DD$-pair $(L,u)$ with $L\cong V_4$ is isomorphic to either $(V_4,1)$ or $(V_4,u)$. One can also show that $\Out(V_4,u)=\{1\}$. Let $\FF_-$ and $V_2$ denote a non-trivial one dimensional module and a two dimensional simple module of $\FF\Out(V_4)\cong \FF\Sym(3)$, respectively.

\begin{theorem}\label{thm Kleinfour}
Let $b$ be a block idempotent of $kG$ with defect groups isomorphic to $V_4$. Then one of the following occurs:

\smallskip 
{\rm (i)} The block idempotent $b$ is nilpotent and $(G,b)$ is functorially equivalent over $\FF$ to $(V_4,1)$. In this case, one has
\begin{align*}
\FFTD_{G,b} \cong S_{1,1,\FF} \oplus 3S_{C_2,1,\FF} \oplus S_{V_4,1,\FF} \oplus S_{V_4,1,\FF_-} \oplus 2S_{V_4,1,V_2}\,.
\end{align*}

\smallskip
{\rm (ii)} The pair $(G,b)$ is functorially equivalent over $\FF$ to $(A_4, 1)$. In this case, one has
\begin{align*}
\FFTD_{G,b} \cong 3S_{1,1,\FF} \oplus S_{C_2,1,\FF} \oplus S_{V_4,1,\FF} \oplus S_{V_4,1,\FF_-1} \oplus 2S_{V_4,u,\FF}\,.
\end{align*}
In particular, the functorial equivalence class of $(G,b)$ depends only on the inertial quotient of $b$.
\end{theorem}
\begin{proof}
It is well-known that if $b$ is a block idempotent of a finite group $G$ with defect groups isomorphic to $V_4$, then $kGb$ is splendidly Rickard equivalent to either $kV_4$ or $kA_4$. Indeed, by \cite{CEKL12}, $kGb$ is splendidly Morita equivalent to $kV_4$, $kA_4$ or $kA_5b_0(A_5)$, and by \cite[Section~3]{Rickard1996} $kA_4$ and $kA_5b_0(A_5)$ are splendidly Rickard equivalent. It follows that $(G,b)$ is functorially equivalent over $\FF$ to either $(V_4,1)$ or $(A_4,1)$. One can find the multiplicities of the simple functors in $\FFTD_{V_4}$ and $\FFTD_{A_4}$ easily using Theorem~\ref{thm multiplicityinrepresentable}.
\end{proof}

\section{Blocks with $Q_8$ defect groups}\label{sec quaternion}

Let $C_4$ denote a cyclic group of order $4$. Since $\Aut(C_4)=\Out(C_4)\cong C_2$ is a $2$-group, the functors $S_{C_4,1,\FF}$ and $S_{C_4,1,\FF_-}$ are the only simple functors, up to isomorphism, with a parametrizing set $(L,u,V)$ with $L\cong C_4$, where $\FF$ and $\FF_-$ denote the trivial and the non-trivial simple $\FF\Out(C_4)$-modules. 

Let $Q_8$ be a quaternion group of order $8$. Let $u\in\Aut(Q_8)\cong \Sym(4)$ be an element of order $3$. One shows that a $\DD$-pair $(L,u)$ with $L\cong Q_8$ is isomorphic to either $(Q_8,1)$ or $(Q_8,u)$. One can also show that $\Out(Q_8,u)=\{1\}$. Indeed, one can show that $\Aut(Q_8\rtimes \langle u\rangle)\cong \Sym(4)$ and $\Inn(Q_8\rtimes \langle u\rangle)\cong \Alt(4)$. Since $Q_8\rtimes \langle u\rangle$ has two conjugacy classes of $3$-elements, but only one automorphism class of $3$-elements, it follows that $\Aut(Q_8,u)=\Inn(Q_8\rtimes \langle u\rangle)$ and hence $\Out(Q_8,u)=\{1\}$. This implies that the simple functors $S_{Q_8,1,\FF}$, $S_{Q_8,1,\FF_-}$, $S_{Q_8,1,V_2}$ and $S_{Q_8,u,\FF}$ are the only simple functors, up to isomorphism, with parametrizing set $(L,u,V)$ with $L\cong Q_8$, where $\FF_-$ and $V_2$ denote the nontrivial one dimensional and the two dimensional simple $\FF\Out(Q_8)\cong \FF\Sym(3)$-modules, respectively.

\begin{theorem}\label{thm quaternion}
Let $b$ be a block idempotent of $kG$ with defect groups isomorphic to $Q_8$. Then one of the following occurs:

\smallskip
{\rm (i)} The block idempotent $b$ is nilpotent and $(G,b)$ is functorially equivalent over $\FF$ to $(Q_8,1)$. In this case, one has
\begin{align*}
\FFTD_{G,b} \cong S_{1,1,\FF}\oplus S_{C_2,1,\FF}\oplus 3S_{C_4,1,\FF} \oplus S_{Q_8,1,\FF} \oplus S_{Q_8,1,\FF_{-1}} \oplus 2S_{Q_8,1,V_2}\,.
\end{align*}

\smallskip
{\rm (ii)} The pair $(G,b)$ is functorially equivalent over $\FF$ to $(SL(2,3),b_0)$, where $b_0$ is the principal $2$-block of $SL(2,3)$. In this case, one has
\begin{align*}
\FFTD_{G,b}\cong 3S_{1,1,\FF}\oplus 3S_{C_2,1,\FF}\oplus S_{C_4,1,\FF} \oplus S_{Q_8,1,\FF} \oplus S_{Q_8,u,\FF} \,.
\end{align*}
\end{theorem}
\begin{proof}
Let $(D,e_D)$ be a maximal $b$-Brauer pair and for any $P\le D$, let $(P,e_P)$ denote the unique $b$-Brauer pair with $(P,e_P)\le (D,e_D)$. Let also $\calF$ denote the fusion system of $b$ with respect to $(D,e_D)$. Up to isomorphism, there are two fusion systems on $Q_8$.

First, assume that $\calF$ is isomorphic to the inner fusion system on $Q_8$. Then the block idempotent $b$ is nilpotent and hence by \cite[Theorem~9.2]{BoucYilmaz2022}, $(G,b)$ is functorially equivalent over $\FF$ to $(Q_8,1)$. Using Theorem~\ref{thm multiplicityinrepresentable} one can easily show that
\begin{align*}
\FFTD_{G,b} \cong \FFTD_{Q_8}\cong S_{1,1,\FF}\oplus S_{C_2,1,\FF}\oplus 3S_{C_4,1,\FF} \oplus S_{Q_8,1,\FF} \oplus S_{Q_8,1,\FF_{-1}} \oplus 2S_{Q_8,1,V_2}\,.
\end{align*}
Now assume that $\calF$ is isomorphic to the non-inner fusion system on $Q_8$. Let $Z$ denote the center of $D$. By \cite[Theorem~3.17]{Olsson1975}, one has $l(kGb)=l(kC_G(Z)e_Z)=3$. Thus, Theorem~\ref{thm multiplicityformula} implies that
\begin{align*}
\Mult(S_{1,1,\FF},\FFTD_{G,b})=\Mult(S_{C_2,1,\FF},\FFTD_{G,b})=3\,.
\end{align*}
We now find the multiplicities of the simple functors $S_{C_4,1,\FF}$ and $S_{C_4,1,\FF_-}$ in $\FFTD_{G,b}$. Let $P$ be a subgroup of $D$ isomorphic to $C_4$. Note that all subgroups of $D$ of order $4$ are $\calF$-isomorphic. The block $kC_G(P)e_P$ has a cyclic defect group $C_D(P)=P$ and so it follows that it has a unique isomorphism class of simple modules. Indetify $P$ with $L=C_4$. One has
\begin{align*}
\calP_{(P,e_P)}(C_4,1)= \Aut(C_4)=C_2
\end{align*}
and
\begin{align*}
[N_G(P,e_P) \backslash \calP_{(P,e_P)}(C_4,1)/ \Aut(C_4)]=[\id]\,.
\end{align*}
It follows that $\Aut(C_4)_{\overline{(C_4,e_{C_4},\id)}}=\Aut(C_4)$. These imply that
\begin{align*}
\Mult(S_{C_4,1,\FF},\FFTD_{G,b})=\dim_\FF\FF^{C_2}=1
\end{align*}
and that
\begin{align*}
\Mult(S_{C_4,1,\FF_-},\FFTD_{G,b})=\dim_\FF(\FF_-)^{C_2}=0\,.
\end{align*}
We finally consider the case $L=Q_8$. Since $\overline{\Aut_{\calF_b}(D)} \cong \Out(Q_8)\cong \Sym(3)$, Lemma~\ref{lem multiplicityoftop} implies that
\begin{align*}
\Mult(S_{Q_8,1,\FF},\FFTD_{G,b})=1 \quad \text{and}\quad \Mult(S_{Q_8,1,\FF_-},\FFTD_{G,b})=\Mult(S_{Q_8,1,V_2},\FFTD_{G,b})=0\,.
\end{align*}
By Theorem~\ref{thm multiplicityformula}, the multiplicity of $S_{Q_8,u,\FF}$ in $\FFTD_{G,b}$ is equal to the cardinality of the set
\begin{align*}
\left[N_G(D,e_D)\backslash \calP_{(D,e_D)}(Q_8,u)/\Aut(Q_8,u)\right]\,.
\end{align*}
One shows that $\calP_{(Q_8,e_{Q_8})}(Q_8,u)=\Aut(Q_8)$ and since $\Aut_{\calF_b}(Q_8)\cong \Aut(Q_8)$, it follows that
\begin{align*}
\Mult(S_{Q_8,u,\FF},\FFTD_{G,b})=1\,.
\end{align*}
This completes the proof. 
\end{proof}

\section{Blocks with $D_8$ defect groups}\label{sec dihedral}

Let $D_8$ be a dihedral group of order $8$. Since $\Aut(D_8)$ is a $2$-group and since $\Out(D)\cong C_2$, the functors $S_{D_8,1,\FF}$ and $S_{D_8,1,\FF_-}$ are the only simple functors, up to isomorphism, with parametrizing set $(L,u,V)$ with $L\cong D_8$, where $\FF$ and $\FF_-$ denote the trivial and the nontrivial simple $\FF\Out(D)\cong \FF C_2$-modules, respectively.

\begin{theorem}\label{thm dihedral}
Let $b$ be a block idempotent of $kG$ with defect groups isomorphic to $D_8$. Then one of the following occurs:

\smallskip
{\rm (i)} The fusion system of $b$ is the inner fusion system on $D_8$. In this case, $b$ is nilpotent and $(G,b)$ is functorially equivalent over $\FF$ to $(D_8,1)$. We have
\begin{align*}
\FFTD_{G,b}\cong S_{1,1,\FF}\oplus 3S_{C_2,1,\FF}\oplus S_{C_4,1,\FF}\oplus 2S_{V_4,1,\FF}\oplus 2S_{V_4,1,V_2}\oplus S_{D_8,1,\FF}\oplus S_{D_8,1,\FF_{-1}}\,.
\end{align*}

\smallskip
{\rm (ii)} The fusion system of $b$ is the non-inner non-simple fusion system on $D_8$. In this case, $(G,b)$ is functorially equivalent over $\FF$ to $(S_4,b_0)$, where $b_0$ is the principal $2$-block of $S_4$. We have
\begin{align*}
\FFTD_{G,b}\cong 2S_{1,1,\FF}\oplus 2S_{C_2,1,\FF}\oplus S_{C_4,1,\FF}\oplus 2S_{V_4,1,\FF}\oplus S_{V_4,1,V_2}\oplus  S_{V_4,u,\FF}\oplus S_{D_8,1,\FF}\oplus S_{D_8,1,\FF_{-1}}\,.
\end{align*}

\smallskip
{\rm (iii)} The fusion system of $b$ is the simple fusion system on $D_8$. In this case, $(G,b)$ is functorially equivalent over $\FF$ to $(PSL(3,2),b_0)$, where $b_0$ is the principal $2$-block of $PSL(3,2)$. We have
\begin{align*}
\FFTD_{G,b}\cong 3S_{1,1,\FF}\oplus S_{C_2,1,\FF}\oplus S_{C_4,1,\FF}\oplus 2S_{V_4,1,\FF}\oplus 2S_{V_4,u,\FF}\oplus S_{D_8,1,\FF}\oplus S_{D_8,1,\FF_{-1}}\,.
\end{align*}
\end{theorem}
\begin{proof}
Let $(D,e_D)$ be a maximal $b$-Brauer pair and for any $P\le D$, let $e_P$ denote the unique block idempotent of $kC_G(P)$ with $(P,e_P)\le (D,e_D)$. Let $\calF$ denote the fusion system of $b$ with respect to $(D,e_D)$. 

Note that up to isomorphism, there are three fusion systems on $D_8$. We denote by $\calF_{00}$ the inner fusion system; by $\calF_{01}$ the non-inner non-simple fusion system; by $\calF_{11}$ the simple fusion system. Note that $\calF_{00}\cong \calF_D(D)$, $\calF_{01}\cong \calF_D(\Sym(4))$ and $\calF_{11}\cong \calF_D(PSL(3,2))$.

By \cite{Brauer1974}, we have $l(kGb)=1$, if $\calF\cong \calF_{00}$; $l(kGb)=2$, if $\calF\cong \calF_{01}$; $l(kGb)=3$, if $\calF\cong \calF_{11}$.  This determines the multiplicity of $S_{1,1,\FF}$ in all cases.

Let $C_2$ be a subgroup of $D$ order $2$. Up to $G$-conjugation, we can assume that $C_2$ is fully $\calF$-centralized, and so the block $kC_G(C_2)e_{C_2}$ has a defect group $C_D(C_2)$ which is isomorphic to $D$ or $V_4$.  In both cases, one can show that $l(kC_G(C_2)e_{C_2})=1$. Therefore, Theorem~\ref{thm multiplicityformula} implies that the multiplicity of $S_{C_2,1,\FF}$ is equal to the number of $\calF$-isomorphism classes of objects isomorphic to $C_2$. Hence
\begin{align*}
\Mult\left(S_{C_2,1,\FF},\FFTD_{G,b}\right)=\begin{cases}
3, & \text{if } \calF\cong \calF_{00}\\
2, & \text{if } \calF\cong \calF_{01}\\
1, & \text{if } \calF\cong \calF_{11}\,.
\end{cases}
\end{align*}

Let $C_4$ be the cyclic subgroup of order $4$ of $D$. The block idempotent $kC_G(C_4)e_{C_4}$ has a central defect group $C_4$ and so $l(kC_G(C_4)e_{C_4})=1$. Moreover, in all cases one has $\Aut_{\calF}(C_4)\cong \Aut(C_4)\cong C_2$. Therefore, Theorem~\ref{thm multiplicityformula} implies that
\begin{align*}
\Mult\left(S_{C_4,1,\FF},\FFTD_{G,b}\right)=\dim_{\FF} \FF^{C_2}=1 \quad \text{and}\quad \Mult\left(S_{C_4,1,\FF_-},\FFTD_{G,b}\right)=\dim_{\FF} \FF_-^{C_2}=0\,.
\end{align*}

Let $X$ and $Y$ be the subgroups of $D$ isomorphic to $V_4$. Note that $X$ and $Y$ are not $\calF$-conjugate. We use the convention that $\Aut_{\calF}(X)\cong \Aut_{\calF}(Y)\cong C_2$ if $\calF\cong \calF_{00}$; $\Aut_{\calF}(X)\cong C_2$ and $\Aut_{\calF}(Y)\cong \Sym(3)$ if $\calF\cong\calF_{01}$; $\Aut_{\calF}(X)\cong \Aut_{\calF}(Y)\cong \Sym(3)$ if $\calF\cong \calF_{11}$. In all cases, the blocks $kC_G(X)e_X$ and $kC_G(Y)e_Y$ has central defect groups $X$ and $Y$, respectively, and hence $l(kC_G(X)e_X)=l(kC_G(Y)e_Y)=1$.

Now let $J\in\{X,Y\}$. First assume that $\Aut_{\calF}(J)\cong C_2$. Then, one has
\begin{align*}
[N_G(J,e_J)\backslash \calP_{(J,e_J)}(V_4,1) / \Aut(V_4)]=[\id]
\end{align*}
and
\begin{align*}
\Aut(V_4)_{\overline{(J,e_J,\id)}}=\{\phi\in \Aut(V_4)|\, \phi=i_g, g\in N_G(J,e_J)\} \cong C_2\,.
\end{align*}
It follows that the $\FF$-dimension of
\begin{align*}
\bigoplus_{\pi \in [\calP_{(J,e_J)}(V_4)]} \FF\Proj(ke_JC_G(J))\otimes_{\Aut(V_4)_{\overline{(J,e_J,\pi)}}} V = \FF\otimes_{C_2}V\cong V^{C_2}
\end{align*}
is equal to one for $V=\FF$ and $V=V_2$ and zero for $V=\FF_{-1}$. Moreover one has
\begin{align*}
\calP_{(J,e_J)}(V_4,u)=\{\phi \in\Aut(V_4): \phi i_u\phi^{-1}\in \Aut_{\calF}(J)\}=\emptyset
\end{align*}
since $ \phi i_u\phi^{-1}$ has order $3$. 

Next, suppose that $\Aut_{\calF}(J)\cong \Sym(3)$. We have
\begin{align*}
[N_G(J,e_J)\backslash \calP_{(J,e_J)}(V_4,1) / \Aut(V_4)]=[\id]
\end{align*}
and
\begin{align*}
\Aut(V_4)_{\overline{(J,e_J,\id)}}=\{\phi\in \Aut(V_4)|\, \phi=i_g, g\in N_G(J,e_J)\} \cong \Sym(3)\,.
\end{align*}
Therefore, the $\FF$-dimension of
\begin{align*}
\bigoplus_{\pi \in [\calP_{(J,e_J)}^G(V_4)]} \FF\Proj(ke_JC_G(J))\otimes_{\Aut(V_4)_{\overline{(J,e_J,\pi)}}} V = \FF\otimes_{\Sym(3)}V\cong V^{\Sym(3)}
\end{align*}
is non-zero only for $V=\FF$. Moreover,
\begin{align*}
\calP_{(J,e_J)}(V_4,u)=\{\phi \in\Aut(V_4)|\, \phi i_u\phi^{-1}\in \Aut_{\calF}(J)\}=\Aut(V_4)\cong S_3
\end{align*}
and
\begin{align*}
[N_G(J,e_J)\backslash \calP_{(J,e_J)}(V_4,u) / \Aut(V_4,u)]=[\id]\,.
\end{align*}
Thus, the $\FF$-dimension of
\begin{align*}
\bigoplus_{\pi \in [\calP_{(J,e_J)}^G(V_4,u)]} \FF\Proj(ke_JC_G(J),u)\otimes_{\Aut(V_4)_{\overline{(J,e_J,\pi)}}} \FF 
\end{align*}
is equal to one. These show that the multiplicities of $S_{V_4,1,\FF}$, $S_{V_4,1,\FF_-}$, $S_{V_4,1,V_2}$ and $S_{V_4,u,\FF}$ in $\FFTD_{G,b}$ are as claimed.

Finally, since in all cases we have $\Aut_{\calF}(D)\cong \Inn(D)$, Lemma~\ref{lem multiplicityoftop} implies that
\begin{align*}
\Mult\left(S_{D_8,1,\FF},\FFTD_{G,b}\right)=\Mult\left(S_{D_8,1,\FF_-},\FFTD_{G,b}\right)=1\,.
\end{align*}
This completes the proof.
\end{proof}

\section{Blocks with $C_2\times C_2\times C_2$ defect groups}\label{sec elementaryabelian}
Let $b$ be a block idempotent of $kG$ with defect groups isomorphic to $D=C_2\times C_2\times C_2$. Let $E$ be the inertial quotient of $b$. Then $E$ has order $1$, $3$, $7$ or $21$. By \cite[Theorem~1]{KessarKoshitaniLinckelmann2012} there is an isotypy, and hence by \cite{Yilmaz2022} a functorial equivalence over $\FF$, between $b$ and its Brauer correspondent block. So we can assume that $D$ is normal in $G$.  Therefore, $b$ has a source algebra $k(D\rtimes E)$. Since moreover, $E$ determines $l(b)$, it also determines the functorial equivalence class over $\FF$ of $b$.  Therefore we have the following.

\begin{theorem}\label{thm elementaryabelian}
Let $b$ be a block idempotent of $kG$ with defect group isomorphic to $C_2\times C_2\times C_2$ and let $E$ be the inertial quotient of $b$. Then one of the following occurs:

\smallskip
{\rm (i)} $|E|=1$ and $(G,b)$ is functorially equivalent over $\FF$ to $(C_2\times C_2\times C_2,1)$.

\smallskip
{\rm (ii)} $|E|=3$ and $(G,b)$ is functorially equivalent over $\FF$ to $(A_4\times C_2,1)$.

\smallskip
{\rm (iii)} $|E|=7$ and $(G,b)$ is functorially equivalent over $\FF$ to $(SL_2(8),b_0(SL_2(8))$.

\smallskip
{\rm (iv)} $|E|=21$ and $(G,b)$ is functorially equivalent over $\FF$ to $(J_1,b_0(J_1))$.

In particular, for blocks with $C_2\times C_2\times C_2$ defect groups, functorial equivalence classes over $\FF$ coincide with isotypy classes. 
\end{theorem}

\section{Blocks with $C_2\times C_4$ defect groups}\label{sec otherabelian}

For completeness, we consider the blocks with defect groups $C_2\times C_4$. Since $\Aut(C_2\times C_4)\cong D_8$, the functors $S_{C_2\times C_4,1,V}$ are the only simple functors with parametrizing set $(L,u,V)$ with $L\cong C_2\times C_4$, where $V\in\{ \FF,\FF_1,\FF_2,\FF_3, V_2\}$ is a simple $\FF D_8$-module.

\begin{theorem}\label{thm abelinoforder8}
Let $b$ be a block idempotent of $kG$ with defect groups isomorphic to $C_2\times C_4$. Then $(G,b)$ is functorially equivalent over $\FF$ to $(C_2\times C_4,1)$. Moreover, one has
\begin{align*}
\Mult\left(S_{1,1,\FF}, \FFTD_{G,b}\right)=1, \quad \Mult\left(S_{C_2,1,\FF}, \FFTD_{G,b}\right)=3,
\end{align*}
\begin{align*}
\Mult\left(S_{C_4,1,V}, \FFTD_{G,b}\right)=2 \quad \text{for } V\in\{\FF,\FF_-\}\,,
\end{align*}
\begin{align*}
\Mult\left(S_{V_4,1,V}, \FFTD_{G,b}\right)=\dim_\FF V \quad \text{for } V\in\{\FF,\FF_-,V_2\}\,,
\end{align*}
\begin{align*}
\Mult\left(S_{C_2\times C_4,1,V}, \FFTD_{G,b}\right)=\dim_\FF V \quad \text{for } V\in\{\FF,\FF_1,\FF_2,\FF_3, V_2\}\,.
\end{align*}
\end{theorem}
\begin{proof}
Since $C_2\times C_4$ has no automorphism of odd order, the block $kGb$ is nilpotent and hence by \cite[Theorem~9.2]{BoucYilmaz2022}, $(G,b)$ is functorially equivalent over $\FF$ to $(C_2\times C_4,1)$. One can find the multiplicities using Theorem~\ref{thm multiplicityinrepresentable}.
\end{proof}

\begin{flushleft}
Deniz Y\i lmaz, Department of Mathematics, Bilkent University, 06800 Ankara, Turkey.\\
{\tt d.yilmaz@bilkent.edu.tr}
\end{flushleft}

\end{document}